\documentclass{amsart}

\usepackage{amssymb, amsmath}
\usepackage{mathrsfs}
\usepackage{amscd}
\usepackage{verbatim}
\usepackage[numbers,sort,square]{natbib}

\usepackage{wasysym}

\usepackage[utf8]{inputenc}
\usepackage{enumerate}
\usepackage{url}

\usepackage[colorinlistoftodos,prependcaption,textsize=tiny]{todonotes}

\usepackage[colorlinks,linkcolor={blue},citecolor={blue},urlcolor={red},]{hyperref}

\allowdisplaybreaks

\theoremstyle{plain}
\newtheorem{theorem}{Theorem}
\theoremstyle{remark}
\newtheorem{remark}[theorem]{Remark}
\newtheorem{example}[theorem]{Example}

\theoremstyle{plain}
\newtheorem{corollary}[theorem]{Corollary}
\newtheorem{lemma}[theorem]{Lemma}
\newtheorem{proposition}[theorem]{Proposition}
\newtheorem{definition}[theorem]{Definition}

\newtheorem{conjecture}[theorem]{Conjecture}

\numberwithin{theorem}{section}

\numberwithin{equation}{section}

\def\R{{\mathbb R}}
\def\C{{\mathbb C}}

\newcommand{\E}{{\mathbb E}}
\renewcommand{\P}{{\mathbb P}}
\newcommand{\A}{{\mathscr A}}
\newcommand{\F}{{\mathscr F}}

\newcommand{\eps}{\varepsilon}

\newcommand{\one}{{{\bf 1}}}

\renewcommand{\emptyset}{\varnothing}

\newcommand{\norm}[1]{||#1||}

\newcommand{\esssup}{{\rm{ess\,sup}}\,}
\newcommand{\essinf}{{\rm{ess\,inf}}\,}

\newcommand{\ud}{\, \mathrm{d}}

\begin{document}

\author{Nick Lindemulder}
\address{Delft Institute of Applied Mathematics\\
Delft University of Technology \\ P.O. Box 5031\\ 2600 GA Delft\\The
Netherlands} \email{N.Lindemulder@tudelft.nl}
\author{Mark Veraar}
\email{M.C.Veraar@tudelft.nl}
\author{Ivan Yaroslavtsev}
\email{I.S.Yaroslavtsev@tudelft.nl}

\dedicatory{Dedicated to Ben de Pagter on the occasion of his 65th birthday}

\title[The UMD property for Musielak--Orlicz spaces]{The UMD property for Musielak--Orlicz spaces}

\begin{abstract}
In this paper we show that Musielak--Orlicz spaces are UMD spaces
under the so-called $\Delta_2$ condition on the generalized Young
function and its complemented function. We also prove that
if the measure space is divisible, then a Musielak--Orlicz space
has the UMD property if and only if it is reflexive. As a
consequence we show that reflexive variable Lebesgue spaces
$L^{p(\cdot)}$ are UMD spaces.

\end{abstract}

\thanks{The first and second named authors are supported by the Vidi subsidy 639.032.427 of the Netherlands Organisation for Scientific Research (NWO)}



\maketitle

\section{Introduction}

The class of Banach spaces $X$ with the UMD (Unconditional Martingale Differences) property is probably the most important one for vector-valued analysis. Harmonic and stochastic analysis in UMD spaces can be found in \cite{Burk01,CPSW,HNVW1,Rub} and references therein. Among other things the UMD property of $X$ implies the following results in $X$-valued harmonic analysis:
\begin{itemize}
\item Marcinkiewicz/Mihlin Fourier multiplier theorems (see \cite{CPSW}, \cite[Theorem 5.5.10]{HNVW1} and \cite[Theorem 8.3.9]{HNVW2});
\item the $T_b$-theorem (see \cite{HytTb});
\item the $L^p$-boundedness of the lattice maximal function (see \cite[Theorem 3]{Rub});
\end{itemize}
and in $X$-valued stochastic analysis:
\begin{itemize}
\item the $L^p$-boundedness of martingale transforms (see
\cite{Burk01, HNVW1});

\item the continuous time Burkholder--Davis--Gundy inequalities (see
\cite{NVW1,VerYarpoint,Yar2018});

\item the lattice Doob maximal $L^p$-inequality (see
\cite{Rub,VerYarpoint}).
\end{itemize}

Most of the classical reflexive spaces are UMD spaces. A list of
known spaces with UMD can be found on \cite[p.\ 356]{HNVW1}. On
the other hand a relatively simple to state space without UMD is
given by $X = L^{p}(L^q(L^{p}(L^q(\ldots)\ldots)))$ with $1<p\neq
q<\infty$ (see \cite{Qiu}). The latter space is not only
reflexive, but also uniformly convex.

In \cite{FG91} and \cite{Liu} it has been shown that an Orlicz
space $L^{\Phi}$ has the UMD property if and only if it is
reflexive. In \cite{FG91} the proof is based on an interpolation
argument and in \cite{Liu} a more direct argument is given which
uses known $\Phi$-analogues of the defining estimates in UMD. The
Musielak-Orlicz spaces of course include all Orlicz spaces but
also the important class of variable Lebesgue spaces
$L^{p(\cdot)}$.

It seems that a study of the UMD property of Musielak-Orlicz
spaces $L^{\Phi}$ and even $L^{p(\cdot)}$ is not available in
the literature yet. In the present paper we show that under
natural conditions on $\Phi$, the Musielak-Orlicz space $L^{\Phi}$
has UMD. We did not see how to prove this by interpolation
arguments and instead we use an idea from \cite{Liu}. Even in the
Orlicz case our proof is simpler, and at the same time it provides more information on the UMD constant.

\begin{theorem}\label{thm:introMainXvalued}
Assume that $\Phi,\Psi:T\times[0,\infty)\to [0,\infty]$ are complementary Young functions which both satisfy the $\Delta_2$ condition. Then the Musielak-Orlicz space $L^{\Phi}(T)$ is a UMD space.
\end{theorem}
This theorem is a special case of Theorem \ref{thm:MainXvalued} below, in which we also have an estimate for the UMD constant in terms of the constants appearing in the $\Delta_{2}$ condition for $\Phi,\Psi$.
The result implies the following new result for the variable Lebesgue spaces.
\begin{corollary}
Assume $1<p_0<p_1<\infty$ and $p:T\to [p_0,p_1]$ is measurable. Then $L^{p(\cdot)}(T)$ is a UMD space.
\end{corollary}

 In the case the measure space is divisible, one can actually
characterize the UMD property in terms of $\Delta_2$ and even in
terms of reflexivity (see Corollary \ref{cor:Orlicz} below).

In the Orlicz setting (i.e.\ $\Phi$ does not dependent on $T$) the
noncommutative analogue of \cite{FG91,Liu} was obtained in
\cite[Corollary 1.8]{DoddSuk}. It would be interesting to obtain
the noncommutative analogues of our results as well. Details on
noncommutative analysis and interpolation theory can be found in
the forthcoming book~\cite{doddstheory}.

\subsubsection*{Notation} For a number $p\in [1, \infty]$ we write $p'\in [1, \infty]$ for its H\"older conjugate which satisfies $\frac{1}{p} + \frac{1}{p'}=1$. For a random variable $f$, $\E(f)$ denotes the expectation of $f$.

\subsection*{Acknowledgment}
The authors would like to thank Emiel Lorist and Jan van Neerven for helpful comments.

\section{Preliminaries}

\subsection{Musielak--Orlicz spaces}
For details on Orlicz spaces we refer to \cite{KrasnoselRu,RGMP}
and references therein. Details on Musielak--Orlicz spaces can be
found in \cite{Musielak,DHHR,Koz,Zaa,Koz77}.

Let $X$ be a Banach space and let $(T, \Sigma,\mu)$ be a $\sigma$-finite measure space. We say that a measurable function $\Phi:T\times [0,\infty)\to [0,\infty]$ is a {\em Young function} if for each $t\in T$,
\begin{enumerate}[(i)]
\item $\Phi(t,0) = 0$, $\exists x_1, x_2>0$ s.t.\ $\Phi(t,x_1)>0$
and $\Phi(t,x_2)<\infty$; \item $\Phi(t,\cdot)$ is increasing,
convex and left-continuous.
\end{enumerate}
As a consequence of the above $\lim_{x\to \infty}\Phi(t,x) = \infty$.

A function $\Phi$ with the above properties is a.e.\ differentiable, the right-derivative $\varphi := \partial_x \Phi$ is increasing and
\[\Phi(t,x) = \int_0^x \varphi(t,\lambda) d\lambda, \ \ \ t\in T, \ x\in \R_+.\]
Note that the function $\varphi(t,\cdot)$ has a right-continuous
version since any increasing function has at most countably
many discontinuities, so $\varphi(t, \lambda) = \lim_{\eps\to 0}
\varphi(t,\lambda+\eps)$ for each $t\in T$ for a.e.\ $\lambda \in
[0,\infty)$.

For a strongly measurable function $f:T\to X$ we say that $f\in
L^{\Phi}(T;X)$ if there exists a $\lambda>0$ such that
\[
\int_T \Phi(t,\|f(t)\|_X/\lambda) \ud\mu(t)<\infty.
\]
The space $L^{\Phi}(T;X)$ equipped with the norm
\begin{equation}\label{eq:defofMOnorm}
\|f\|_{L^{\Phi}(T;X)}:= \inf\Big\{\lambda>0:\int_T
\Phi(t,\|f(t)\|_{X}/\lambda) \ud\mu(t)\leq 1\Big\}
\end{equation}
is a Banach space. Here as usual we identify functions which are almost everywhere identical.
The space $L^{\Phi}(T;X)$ is called the $X$-valued {\em Musielak-Orlicz space} associated with $\Phi$.

The following norm will also be useful in the sequel.
\begin{equation}\label{eq:Orliczequiv}
\|f\|_{X,\Phi}:=\inf_{\lambda>0} \frac{1}{\lambda} \Big[1+\int_{T} \Phi(t,\lambda \|f(t)\|_X) \ud\mu(t)\Big].
\end{equation}
It is simple to check that this gives an equivalent norm (see \cite[Lemma 2.1]{HytVer})
\begin{equation}\label{eq:OrliczLux}
\|f\|_{L^\Phi(T;X)}\leq \|f\|_{X,\Phi} \leq 2\|f\|_{L^\Phi(T;X)}.
\end{equation}
In case $X = \R$ or $X = \C$, we write $L^{\Phi}(T)$ for the above space.

\begin{example}\label{ex:Young_function_variable_Lebesgue}
Let $p:T\to [1, \infty)$ be a measurable function and let
$\Phi(t,\lambda) = |\lambda|^{p(t)}$. Then
$L^{\Phi}(T)$ coincides with the variable Lebesgue space
$L^{p(\cdot)}$.
\end{example}

Next we recall condition $\Delta_2$ from \cite[Theorem
8.13]{Musielak}. There it was used to study the dual space of the
Musielak--Orlicz space and to prove uniform convexity and in
particular reflexivity (see \cite[Section 11]{Musielak}).
Let $L^1_{+}(T)\subset L^1(T)$ be the set of all nonnegative
integrable functions.
\begin{definition}
A Young function $\Phi:T\times [0,\infty)\to [0,\infty]$ is said
to be in $\Delta_2$ if there exists a $K>1$ and an $h\in
L^1_+(T)$ such that for a.a.\ $t \in T$
\[
\Phi(t,2\lambda) \leq K\Phi(t,\lambda) + h(t), \ \ \lambda\in [0,\infty).
\]
\end{definition}
Note that $\Phi\in \Delta_2$ implies that $\Phi(t,\lambda)<\infty$ for almost all $t\in T$ and all $\lambda\in [0,\infty)$.
Unlike is standard for Young's function independent of $T$,
the condition $\Delta_2$ depends on the measure space; namely, if one has that
$\mu(T)=\infty$ and $h$ does not depend on $t\in T$, then $h=0$.

If $\Phi:T\times [0,\infty)\to [0,\infty]$ is a Young function we define its {\em complemented function} $\Psi:T\times [0,\infty)\to [0,\infty]$ by the Legendre transform
\[\Psi(t,x) = \sup_{y\geq 0} (xy - \Phi(t,y)).\]
Then $\Psi$ is a Young function as well. Moreover, one can check that the complemented function of $\Psi(t,\cdot)$ equals $\Phi(t,\cdot)$.

\begin{example}\label{ex:Young_function_variable_Lebesgue;Delta2&complemented_function}
Let the notations be as in Example~\ref{ex:Young_function_variable_Lebesgue}. Then the following statements hold.
\begin{enumerate}[(i)]
\item\label{it:ex:Young_function_variable_Lebesgue;Delta2&complemented_function;Delta2} $\Phi$ is in $\Delta_{2}$ if and only if $p \in L^{\infty}(T)$, in which case $\Phi$ satisfies the $\Delta_{2}$-condition with $K=2^{\norm{p}_{\infty}}$ and $h=0$.
\item The complemented function $\Psi$ to $\Phi$ is given by
\[
\Psi(t,x) = x^{p'(t)}\one_{\{p>1\} \times [0,\infty)}(t,x) + \infty\cdot \one_{\{p=1\} \times (1,\infty)}(t,x),
\]
where $p'(t)=p(t)'$ is the H\"older conjugate.
\end{enumerate}
In particular, $\Phi$ and $\Psi$ are both in $\Delta_{2}$ if and only if $1< \essinf p \leq \esssup p < \infty$, in which case $\Psi(t,x) = x^{p'(t)}$ for a.a.\ $t \in T$ and all $x \in [0,\infty)$.
\end{example}
\begin{proof}
Let us only give the proof of \eqref{it:ex:Young_function_variable_Lebesgue;Delta2&complemented_function;Delta2}.
If $p \in L^{\infty}(T)$, then, for a.a.\ $t \in T$,
\[
\Phi(t, 2\lambda) = 2^{p(t)} \Phi(t,\lambda) \leq 2^{\|p\|_{\infty}} \Phi(t,\lambda), \;\;\; \lambda \in [0,\infty).
\]
Conversely assume that $\Phi$ is in $\Delta_{2}$. Let $K$ and $h$ be as in the $\Delta_{2}$ condition for $\Phi$. Then, for a.a.\ $t \in T$ and all $\lambda \in [0,\infty)$,
\[
2^{p(t)}\Phi(t,\lambda) = \Phi(t, 2\lambda) \leq K\Phi(t,\lambda) + h(t)
\]
and thus
\[
(2^{p(t)}-K)\Phi(t,\lambda) \leq  h(t).
\]
As $\lim_{\lambda \to \infty}\Phi(t,\lambda) = \infty$, this implies that $2^{p(t)} \leq K$ for a.a.\ $t \in T$. Hence, $p \in L^{\infty}(T)$.
\end{proof}

By the properties of the functions $\Phi$ and $\Psi$ one can check
that for $\varphi = \partial_x\Phi$ and $\psi=\partial_x \Psi$
(where $\varphi$ and $\psi$ are taken right-continuous in $x$), we have
$\varphi^{-1}(t,\cdot) = \psi(t,\cdot)$, where
\begin{equation}\label{eq:rightcont}
\varphi^{-1}(t,y)  =   \sup\{x: \varphi(t,x)\leq y\}\;\;\; y\geq
0.
\end{equation}
Note that $\psi(t,\varphi(t,x))\geq x$ and $\varphi(t,\psi(t,x))\geq x$ because of the above choices.

Recall Young's inequality (see \cite[Section I.2]{KrasnoselRu}  or \cite[Proposition 15.1.2]{RGMP}) for a.a.\ $t\in T$,
\begin{equation}\label{eq:Young}
xy \leq \Phi(t,x) + \Psi(t,y), \ \ x,y\geq0
\end{equation}
with equality if and only if $y = \varphi(t,x)$ or $x = \psi(t,y)$.

\begin{lemma}\label{lem:tobeused}
Let $\Phi:T\times[0,\infty)\to [0,\infty)$ be a Young function and
let $\Psi$ be its complemented function. If $\Phi\in \Delta_2$ with constant $K>1$ and $h\in L_{+}^1(T)$, then for almost all $t\in T$,
\[\Psi(t,\lambda) \leq \frac{K}{K-1} \lambda \psi(t,\lambda) + \frac{1}{K}h(t), \ \ \lambda\geq 0.\]
\end{lemma}
\begin{proof}
We use a variation of the argument in \cite[Section
1.4]{KrasnoselRu}. By the $\Delta_2$ condition there exist $K>1$
and $h\in L^1_{+}(T)$ such that for almost all $t\in T$ and all $\lambda \geq
0$
\[
K\Phi(t,\lambda) + h(t)\geq \Phi(t,2\lambda) = \int_0^{2\lambda}
\varphi(t,x) dx \geq \int_{\lambda}^{2\lambda} \varphi(t,x) dx
\geq \lambda\varphi(t,\lambda),
\]
where we used the fact that
$\varphi(t,\cdot)$ is increasing. Using the identity case of
\eqref{eq:Young} we obtain
\[
K\lambda \varphi(t,\lambda) - K\Psi(t,\varphi(t,\lambda))+ h(t)
\geq \lambda\varphi(t,\lambda)
\]
Therefore,
\[\frac{K\Psi(t,\varphi(t,\lambda))}{\lambda\varphi(t,\lambda)} \leq K-1 + \frac{h(t)}{\lambda\varphi(t,\lambda)}.\]
Taking $\lambda = \psi(t,x)$ and using the estimates below
\eqref{eq:rightcont} and the fact that $y\mapsto
\frac{\Psi(t,y)}{y}$ is increasing (see \cite[(1.18)]{KrasnoselRu}) we
obtain
\begin{align*}
\frac{K \Psi(t,x)}{x \psi(t,x)} & \leq \frac{K\Psi(t,\varphi(t,\psi(t,x)))}{\psi(t,x)\varphi(t,\psi(t,x))}
\\ & \leq K-1 + \frac{h(t)}{\psi(t,x) \varphi(t,\psi(t,x))}\leq K-1 + \frac{h(t)}{x \psi(t,x)}.
\end{align*}
We may conclude that
\[\Psi(t,x)\leq \frac{K-1}{K} x \psi(t,x) + \frac{1}{K} h(t). \qedhere\]
\end{proof}

\subsection{UMD spaces}
For details on UMD spaces the reader is referred to
\cite{Burk01,Rub} and the monographs \cite{HNVW1,HNVW2}. Let
$(\Omega, \mathcal{A}, (\F_n)_{n\geq 0}, \P)$ denote a filtered
probability space which is rich enough in the sense that it supports an i.i.d.\ sequence $(\varepsilon_n)_{n\geq 0}$ such that
$\P(\varepsilon_n=1) = \P(\varepsilon_{n}=-1)=\frac12$ for each $n \geq 0$. Such a
sequence is called a Rademacher sequence.

For a sequence of random variables $f=(f_n)_{n\geq 0}$ with values
in $X$, we write $f_n^* = \sup_{k\leq n} \|f_k\|_X$ and $f^* =
\sup_{k\geq 0} \|f_k\|_X$. Moreover, if $\epsilon = (\epsilon_n)_{n\geq 0}$ is a sequence
of signs, we write $(\epsilon*f)_n = \sum_{k=0}^n \epsilon_k(f_k -
f_{k-1})$, where $f_{-1} = 0$.

We say that $X$ is a UMD space if there exists a $p\in (1, \infty)$ and $\beta\in [1, \infty)$ such that for all $L^p$-martingales $f = (f_n)_{n\geq 0}$ and all sequences of signs $\epsilon = (\epsilon_n)_{n\geq 0}$ we have that
\[
\|(\epsilon*f)_n\|_{L^p(\Omega;X)}\leq \beta \|f_n\|_{L^p(\Omega;X)}, \ \ \  n\geq 0,
\]
where the least admissible constant $\beta$ is denoted by $\beta_{p,X}$ and is called {\em the UMD constant}.
If the above holds for some $p\in (1, \infty)$, then it holds for all $p\in (1, \infty)$. Examples and counterexamples of UMD spaces have been mentioned in the introduction. Every UMD space is (super-)reflexive (see \cite[Theorem 4.3.8]{HNVW1}).

We say that $f=(f)_{n \geq 0}$ is a {\em Paley--Walsh martingale} if $f$ is a
martingale with respect to the filtration $(\F_n)_{n\geq 0}$ with $\F_0 = \{\emptyset, \Omega\}$ and
$\F_n = \sigma\{\varepsilon_k:1\leq k\leq n\}$ for some Rademacher sequence
$(\varepsilon_k)_{k\geq 0}$ and if $f_0 =0$.

The following result follows from \cite[Theorems 1.1 and 3.2]{Burk860}.
\begin{proposition}\label{prop:chUMD}
Let $X$ be a Banach space. Then $X$ is a UMD space if and only if
for all Paley--Walsh martingales $f$ and all sequences of signs $\epsilon$ we
have
\[\sup_{n\geq 0}\|f_n\|_{L^\infty(\Omega;X)}<\infty \ \Longrightarrow \ \P(\sup_{n\geq 0}\|(\epsilon*f)_n\|_X<\infty) >0.\]
\end{proposition}

We will also need the following lemma which allows to estimate the
$\epsilon$-transform for different functions than $\Phi(x) =
|x|^p$. This lemma is a straightforward extension \cite[p.\
1001]{Bu1} where the case $b=0$ was considered. Moreover, since we
only state it for Paley--Walsh martingales it follows from
\cite[Proof of (10)]{Burk01}.

\begin{lemma}\label{lem:UMDPhi}
Assume $X$ is a UMD space. Let $\Phi:[0,\infty)\to [0,\infty)$ be a Young function and
assume that there exist constants $K>1$ and $b\geq 0$ such that
\begin{equation}\label{eq:lemUMDPhi}
\Phi(2\lambda) \leq K\Phi(\lambda) + b,\;\;\;\lambda\geq 0.
\end{equation}
Let $f=(f_{n})_{n \geq 0}$ be a Paley--Walsh martingale,
$\epsilon=(\epsilon_n)_{n\geq 0}$ be a sequence of signs, and set
$g := \epsilon *f$. Then there exists a constant $C_{K,X}\geq 0$
only depending on $K$ and (the UMD constant of) $X$ such that
\[
\E \Phi(g^*)\leq C_{K,X} (\E\Phi(f^*) + b).
\]
\end{lemma}

\begin{remark}
To obtain Lemma \ref{lem:UMDPhi} in the case of general
martingales (as it is done in \cite[p.\ 1001]{Bu1}), one can use
the Davis decomposition to reduce to a bad part and a good part of
$f$. To estimate the bad part of the Davis decomposition one can
use \cite[Theorem 3.2 and the proof of Theorem 2.1]{BDG} (see
\cite[Proposition A-3-5]{Nev} and \cite[Theorem 53]{Mey} for a
simpler proof).
\end{remark}

Recall that $X$ is a UMD space if and only if it is {\em
$\zeta$-convex}, i.e.\ there exists a biconvex function $\zeta:X
\times X \to \mathbb R$ such that $\zeta(0,0)>0$ and $\zeta(x, y)
\leq \|x+y\|$ for all $x, y\in X$ with $\|x\| = \|y\|=1$ (see
\cite{HNVW1,Burk860,Bu83}). By the $\zeta$-function we will usually
mean the {\em optimal $\zeta$-function} which can be defined as the
supremum over all admissible $\zeta$'s, and this obviously satisfies
the required conditions.

The following theorem can be found in \cite[equation (20)]{Burk01}.
\begin{theorem}(Burkholder)\label{rem:zetafunction}
Let $X$ be a UMD Banach space and let $\zeta:X\times X \to \mathbb R$ be an optimal $\zeta$-function (i.e.\ $\zeta(0,0)$ is maximal). For any $1<p<\infty$ one then has that
\begin{equation}\label{eq:betavszeta}
\frac{1}{\zeta(0,0)} \leq \beta_{p, X} \leq \frac{72}{\zeta(0,0)}
\frac{(p+1)^2}{p-1}.
\end{equation}
\end{theorem}

The following lemma follows from \cite[p.\ 49]{Burk860}.

\begin{lemma}\label{lem:zeta00approxbymart}
Let $X$ be a UMD Banach space and let $\zeta:X\times X \to \mathbb R$ be an optimal $\zeta$-function (i.e.\ $\zeta(0,0)$ is maximal). Then for any $\eps>0$ there exist an
$X$-valued Paley--Walsh martingale $f=(f_n)_{n\geq 1}$ which starts
in zero and a sequence of signs $\epsilon = (\epsilon_n)_{n\geq 1}$ such that
$\mathbb P (g^*>1) = 1$ and $\sup_{n\geq 1}\mathbb E\|f_n\| \leq
\tfrac{\zeta(0,0)}{2} + \eps$, where $g := \epsilon * f$.
\end{lemma}

\begin{remark}\label{rem:oderofCKX}
Let us compute  an upper bound for $C_{K,X}$ in Lemma \ref{lem:UMDPhi}.
Let $M\geq 1$ be the least integer such that $2^{-M} \leq \tfrac{\zeta(0,0)}{48 K}$. Fix $\beta := 2$ and $\delta := 2^{-M}$. Then by formula \cite[(1.8)]{Bu1} one has that for $f$ and $g$ from Lemma \ref{lem:UMDPhi}
\[
\mathbb P(g^* > 2\lambda, f^* \leq 2^{-M} \lambda) \leq \eps \mathbb P(g^* > \lambda),\;\;\; \lambda >0,
\]
where $\eps = 3 c \delta/(\beta-\delta-1) \leq 1/(2K)$, and where we used the fact that the constant $c$ from \cite[(1.2)]{Bu1} can be bounded from above by $4/\zeta(0,0)$ by \cite[Theorem 3.26 and Lemma 3.23]{Osbook} (see also \cite{Y19weakL1}).
Note that by \eqref{eq:lemUMDPhi}
\[
\Phi(\beta \lambda) \leq K \Phi(\lambda) + b,\;\;\; \Phi(\delta^{-1} \lambda) \leq K^M \Phi(\lambda) + b M K^M,\;\;\; \lambda>0,
\]
where one needs to iterate \eqref{eq:lemUMDPhi} $M$ times in order to get the latter inequality. Therefore by exploiting \cite[proof of Lemma 7.1]{Burk73} one has the following analogue of the formula \cite[(7.6)]{Burk73}
\[
\mathbb E \Phi(2^{-1} g^*) \leq \eps \mathbb E \Phi (g^*)  + K^M \mathbb E\Phi(f^*) + b M K^M,
\]
and by using the fact that $\mathbb E \Phi( g^*) \leq K \mathbb E \Phi(2^{-1}  g^*) + b$ and the fact that $\eps K \leq 1/2$ one has that
\[
\mathbb E \Phi(g^*) \leq 2K^{M+1} \mathbb E \Phi(f^*) + 2b(1+MK^{M+1}) \leq 2(MK^{M+1} + 1)(\mathbb E \Phi(f^*)+b),
\]
so $C_{K, X} \leq 2(MK^{M+1} + 1)$, where $M$ can be taken $[\log_2 \tfrac{48K}{\zeta(0,0)}] +1$. Of course this bound is not optimal.
\end{remark}

\section{Musielak-Orlicz spaces are UMD spaces}

The main result of this paper is the following.

\begin{theorem}\label{thm:MainXvalued}
Assume $X$ is a UMD space. Let $\Phi,\Psi:T\times[0,\infty)\to
[0,\infty)$ be complemented Young functions which both satisfy
$\Delta_2$. Then the Musielak-Orlicz space $L^{\Phi}(T;X)$ is a
UMD space.

Moreover, if $\Phi \in \Delta_2$ with constant $K_{\Phi}$ and
$h_{\Phi} \in L^1_{+}(T)$ and $\Psi \in \Delta_2$ with constant
$K_{\Psi}$ and $h_{\Psi} \in L^1_{+}(T)$, then for the optimal
$\zeta$-function $\zeta: L^{\Phi}(T;X) \times L^{\Phi}(T;X) \to
\mathbb R$ (see the discussion preceding Theorem~\ref{rem:zetafunction}) one has that
\begin{equation}\label{eq:lowerboundforzeta00}
\zeta(0,0) \geq \frac{1}{6K_{\Psi}C_{K_{\Phi}, X}C_h},
\end{equation}
and
\begin{equation}\label{eq:upperboundforbeta}
\beta_{p, L^{\Phi}(T;X)} \leq 432K_{\Psi}C_{K_{\Phi}, X}C_h\frac{(p+1)^2}{p-1},
\end{equation}
where $C_{K_{\Phi}, X}$ is as in Lemma \ref{lem:UMDPhi} and $C_h := 2 + \|h_{\Phi}\|_{L^1(T)} + \tfrac{1}{K_{\Psi}}\|h_{\Psi}\|_{L^1(T)}$.
\end{theorem}

This result is well-known in the case of $\Phi(x) = |x|^p$, and
then it is a simple consequence of Fubini's theorem which allows
to write $L^p(\Omega;L^p(T;X)) = L^p(T;L^p(\Omega;X))$ and to
apply the UMD property of $X$ pointwise a.e.\ in $T$ (see
\cite[Proposition 4.2.15]{HNVW1}). Such a Fubini argument is
necessarily limited to $L^p$-spaces. Indeed, the
Kolmogorov–Nagumo theorem says that for Banach function spaces
$E$ and $F$ one has $E(F) = F(E)$  isomorphically, if and only if
$E$ and $F$ are weighted $L^p$-spaces (see \cite[Theorem
3.1]{BBS}).

To prove Theorem \ref{thm:MainXvalued} we will use several results from the preliminaries. Moreover, we will need the following scalar-valued result which is a well-known version of Doob's maximal inequality for a certain class of Young functions.
\begin{proposition}\label{prop:DoobPhi}
Suppose that $\Phi:[0,\infty)\to [0,\infty]$ is a Young function with a
right-continuous derivative $\varphi:[0,\infty)\to [0,\infty)$ and
that there exists a $q\in (1, \infty)$ and $c\in [0,\infty)$ such
that
\[
\Phi(\lambda) \leq \frac{1}{q} \lambda \varphi(\lambda)+c,\;\;\;
\lambda \geq 0.
\]
Then for all nonnegative submartingales $(f_n)_{n\geq 0}$
\[
\E\Phi(f_n^*)\leq \E \Phi(q' f_n) + c,\;\;\; n\geq 0.
\]
In particular, $\|f_n^*\|_{\Phi} \leq  q' (1+c) \|f_n\|_{\Phi}$.
\end{proposition}

\begin{proof}
The result for $c=0$ is proved in \cite[estimate (104.5)]{DM}, and the case $c>0$ follows by a simple modification of that argument.
The final assertion follows from the obtained estimate since for any $\lambda>0$ we have
\begin{align*}
\|f_n^*\|_{\Phi} & \leq \lambda^{-1}(1+\E\Phi(\lambda f_n^*))
\\ & \leq \lambda^{-1}(c+1+\E \Phi(\lambda q' f_n))
\\ & \leq (c+1) \lambda^{-1}(1+\E \Phi(\lambda q' f_n)).
\end{align*}
Taking the infimum over all $\lambda>0$ yields the required conclusion.
\end{proof}

\begin{proof}[Proof of Theorem \ref{thm:MainXvalued}]
Let $Y:=L^{\Phi}(T;X)$.
In order to prove the theorem we will use Proposition \ref{prop:chUMD}. Let $f=(f_{n})_{n \geq 0}$ be a Paley--Walsh martingale with values in $Y$. Let $\epsilon = (\epsilon_n)_{n\geq 0}$ be a sequence of signs and $g := \epsilon*f$. We will show that $g^*<\infty$ a.s. For this it is enough to show that
\begin{equation}\label{eq:toprove}
\E \sup_{n\geq 0}\|g_n\|_Y \leq K_{\Psi}C_{K_{\Phi},X}C_{h} \sup_{n\geq
0}\|f_n\|_{L^\infty(\Omega;Y)},
\end{equation}
where $C_{K_{\Phi}, X}$ is as in Lemma \ref{lem:UMDPhi} and $C_h = 2 + \|h_{\Phi}\|_{L^1(T)} + \tfrac{1}{K_{\Psi}}\|h_{\Psi}\|_{L^1(T)}$.
By homogeneity we can assume $\sup_{n\geq
0}\|f_n\|_{L^\infty(\Omega;Y)}=1$.

We know that $K_{\Phi}>1$ and a function $h_{\Phi}\in L^1_+(T)$
satisfy the following inequality
\begin{equation}\label{eq:PhiDelta}
\Phi(t,2\lambda) \leq K_{\Phi}\Phi(t,\lambda)+h_{\Phi}(t), \ \ \
\lambda\in [0,\infty), \ t\in T.
\end{equation}
Since $\Psi$ satisfies $\Delta_2$ with constant $K_{\Psi}>1$ and
$h_{\Psi} \in L^1_{+}(T)$ it follows from Lemma \ref{lem:tobeused}
that
\begin{equation}\label{eq:Phinabla}
\Phi(t,\lambda)\leq \frac{K_{\Psi}-1}{K_{\Psi}} \lambda
\varphi(t,\lambda)+ \frac{1}{K_{\Psi}}h_{\Psi}(t), \ \ \lambda\in
[0,\infty), \ \ t\in T.
\end{equation}

One can check that for a.e.\ $t\in T$, $f(t)$ is an $X$-valued
martingale and  $g_n(t) = (\epsilon*(f(t)))_n$ (use that $f$ is a Paley--Walsh martingale). Therefore, first applying \eqref{eq:PhiDelta} and
Lemma \ref{lem:UMDPhi} and then \eqref{eq:Phinabla} and
Proposition \ref{prop:DoobPhi} to the submartingale
$(\|f_k(t)\|_X)_{k\geq 0}$ gives that for almost all $t\in T$,
\begin{align*}
\E \Phi(t,\sup_{k\leq n}\|g_k(t)\|_X)  & \leq C_{K_{\Phi},X} \left[ \E
\Phi(t,\sup_{k\leq n}\|f_k(t)\|_X) + h_{\Phi}(t) \right]
\\  & \leq C_{K_{\Phi},X} \left[ \E \Phi(t,K_{\Psi}\|f_n(t)\|_X) + h_{\Phi}(t)+\tfrac{1}{K_{\Psi}}h_{\Psi}(t) \right].
\end{align*}
The same holds with $(f,g)$ replaced by $(\lambda f, \lambda g)$ for any $\lambda>0$. Integrating over $t\in T$ (and using \eqref{eq:Orliczequiv}) we find that
\begin{align*}
\E & \sup_{k\leq n} \|g_k\|_{X,\Phi} \leq \E \sup_{k\leq n}
\frac{1}{\lambda}\Big(1+\int_T \Phi(t,\lambda \|g_k(t)\|_X)
d\mu(t) \Big)
\\ & \stackrel{(*)}\leq \E  \frac{1}{\lambda}\Big(1+\int_T \Phi(t,\sup_{k\leq n} \lambda \|g_k(t)\|_X) d\mu(t) \Big)
\\ & \leq C_{K_{\Phi},X} \E  \frac{1}{\lambda}\Big(1+\int_T \Phi(t,K_{\Psi}\lambda \|f_n(t)\|_X) d\mu(t) +\|h_{\Phi}\|_{L^1(T)}+\|\tfrac{1}{K_{\Psi}}h_{\Psi}\|_{L^1(T)}  \Big),
\end{align*}
where $(*)$ follows form the fact that $\sup \int \leq \int \sup$
and the fact that the map $\lambda \mapsto \Phi(t,\lambda)$ is
increasing in $\lambda \geq 0$.
Since $\int_T \Phi(t,\|f_n(t)\|_X) d\mu(t)\leq 1$ a.s.\ by
\eqref{eq:defofMOnorm} and the assumption
$\|f_{n}\|_{L^{\infty}(\Omega;Y)}\leq 1$, it follows by setting $\lambda =
1/K_{\Psi}$ that
\begin{equation*}\label{eq:g*isbddbyinftymomentoff*}
\E \sup_{k\leq n} \|g_k\|_{X,\Phi} \leq K_{\Psi}C_{K_{\Phi},X}C_h.
\end{equation*}
Now the required estimate \eqref{eq:toprove} follows from
\eqref{eq:OrliczLux}.

For proving \eqref{eq:lowerboundforzeta00} and
\eqref{eq:upperboundforbeta} we will use Lemma~\ref{lem:zeta00approxbymart}.
By the first part of the proof, $Y=L^{\Phi}(T;X)$ is UMD.
Fix $\eps>0$. Then by Lemma~\ref{lem:zeta00approxbymart}
there exist a $Y$-valued Paley--Walsh martingale $f=(f_n)_{n\geq 0}$
which starts in zero and a sequence of signs $\epsilon = (\epsilon_n)_{n\geq
0}$ such that $\mathbb P (g^*>1) = 1$ and $\sup_{n\geq 0}\mathbb
E\|f_n\|_{Y} \leq \tfrac{\zeta(0,0)}{2} + \eps$, where $g := \epsilon * f$. By \cite[Lemma
3.1]{Burk860} there exist discrete $Y$-valued Paley--Walsh martingales $F=(F_n)_{n\geq 0}$ and $G=(G_n)_{n\geq 0}$ such that $G = \epsilon*F$, $\mathbb P(G^* >1) \leq
1/2$, and
$$
\sup_{n\geq 0}\|F_n\|_{L^{\infty}(\Omega;Y)} \leq 6\sup_{n\geq
0}\mathbb E\|f_n\|_{Y}.
$$
Therefore, by \eqref{eq:toprove},
\begin{multline*}
\frac 12 \leq \mathbb E G^* \leq K_{\Psi}C_{K_{\Phi},X}C_h \sup_{n\geq
0}\|F_n\|_{L^{\infty}(\Omega;Y)}\\
 \leq 6K_{\Psi}C_{K_{\Phi},X}C_h\sup_{n\geq
0}\mathbb E\|f_n\|_{Y} \leq 3K_{\Psi}C_{K_{\Phi},X}C_h (\zeta(0,0) +2\eps),
\end{multline*}
so letting $\eps\to 0$ gives \eqref{eq:lowerboundforzeta00}.
\eqref{eq:upperboundforbeta} follows from
\eqref{eq:lowerboundforzeta00} and \eqref{eq:betavszeta}.
\end{proof}

We recover the following result of \cite{FG91} and \cite{Liu}.
Recall that a measure space $(T,\Sigma,\mu)$ is divisible if for
every $A\in \Sigma$ and $t\in (0,1)$ there exist sets $B,C\in
\Sigma$ such that $B,C\subseteq A$, $\mu(B) = t\mu(A)$ and $\mu(C) = (1-t)\mu(A)$.
The divisibility condition is only needed in the implication
(ii)$\Rightarrow$(iii).

\begin{corollary}\label{cor:Orlicz}
Let $X \neq \{0\}$ be a Banach space and assume that $T$ is divisible and
$\sigma$-finite. Suppose $\Phi,\Psi:T \times [0,\infty)\to [0,\infty]$ are
complementary Young functions. Then the following are equivalent:
\begin{enumerate}[(i)]
\item $L^{\Phi}(T;X)$ is a UMD space; \item $L^{\Phi}(T)$ is
reflexive and $X$ is a UMD space; \item $\Phi$ and $\Psi$ both
satisfy $\Delta_2$ and $X$ is a UMD space.
\end{enumerate}
\end{corollary}

For the proof we will need the following lemma which follows from
\cite[Theorem 2.2]{Koz} and \cite[Theorem 4.7]{Koz77}.

\begin{lemma}\label{lem:dualofLPhi}
Let $\Phi, \Psi:T \times [0,\infty)\to [0,\infty]$ be
complementary Young functions. Then there exists a decomposition
$L^{\Phi}(T)^* = L^{\Psi}(T) \oplus \Lambda$ of the dual of
$L^{\Phi}(T)$ into a direct sum of two Banach spaces, where $g\in
L^{\Psi}(T)$ acts on $L^{\Phi}(T)$ in the following way:
\[
\langle f, g\rangle = \int_T fg \ud \mu,\;\;\; f\in L^{\Phi}(T).
\]
\end{lemma}

\begin{proof}[Proof of Corollary \ref{cor:Orlicz}]
{\em (i) }$\Rightarrow${\em (ii) }: Fix $h\in L^{\Phi}(T)$ and $x \in X$ of norm
one. Then $L^{\Phi}(T)$ and $X$ can be identified with the closed subspaces $L^{\Phi}(T) \otimes x$ and $h \otimes X$ of the UMD space $L^{\Phi}(T;X)$, respectively, and therefore have UMD themselves.
In particular, $L^{\Phi}(T)$ is reflexive.

{\em (ii) }$\Rightarrow${\em (iii) }: We show that $\Phi$
satisfies $\Delta_2$. The proof for $\Psi$ is similar. By Lemma~\ref{lem:dualofLPhi}, $L^{\Phi}(T)^* = L^{\Psi}(T) \oplus
\Lambda$, so
$$
L^{\Phi}(T)^{**} = L^{\Psi}(T)^{*} \oplus \Lambda^{*} \supseteq
L^{\Phi}(T) \oplus \Lambda^{*},
$$
where the latter inclusion follows from Lemma \ref{lem:dualofLPhi}
and means that $L^{\Phi}(T) \oplus \Lambda^{*}$ is a closed
subspace of $L^{\Psi}(T)^{*} \oplus \Lambda^{*}$, and hence of
$L^{\Phi}(T)^{**}$. Thus $\Lambda = 0$ due to the reflexivity of
$L^{\Phi}(T)$, and since $T$ is divisible the desired statement follows from \cite[Corollary
1.7.4]{Koz}.

{\em (iii) }$\Rightarrow${\em (i) }:  This follows from Theorem
\ref{thm:MainXvalued}.
\end{proof}

As a consequence of the above results many other spaces are UMD as
well. Indeed, it suffices to be isomorphic to a closed subspace
(or quotient space) of an $L^{\Phi}(T;X)$ space with UMD. This
applies to the Musielak--Orlicz variants of Sobolev, Besov, and
Triebel--Lizorkin spaces.

\begin{remark}\label{rem:EX}
A result of Rubio de Francia (see \cite[p.\ 214]{Rub}) states that
for a Banach function space $E$ and a Banach space $X$ one has
that $E(X)$ is a UMD space if and only if $E$ and $X$ are both UMD
spaces. Therefore, it actually suffices to consider $X = \R$ in
the proof of Theorem \ref{thm:MainXvalued}. Since our argument
works in the vector-valued case without difficulty, we consider
that setting from the start.
\end{remark}

 For the variable Lebesgue spaces we obtain the following
consequence. For a measurable mapping $p:T\to [1, \infty]$ we will write $p_+
= \|p\|_{L^{\infty}(T)}$ and $p_- = \|1/p\|_{L^\infty(T)}^{-1}$.

\begin{corollary}\label{cor:Lp}
Let $X\neq \{0\}$ be a Banach space and assume $T$ is divisble and $\sigma$-finite. Assume $p:T\to [1,\infty]$ is measurable. Then the following assertions are equivalent.
\begin{enumerate}[(i)]
\item $L^{p(\cdot)}(T;X)$ is a UMD space;
\item $L^{p(\cdot)}(T)$ is reflexive and $X$ is a UMD space;
\item $p_->1$ and $p_+<\infty$ and $X$ is a UMD space.
\end{enumerate}
\end{corollary}

The result that $L^{p(\cdot)}(T)$ is reflexive if and only if $p_->1$ and $p_+<\infty$ can also be found in \cite[Proposition~2.79$\&$Corollary~2.81]{Cruz-Uribe&Fiorenza2013} and \cite[Remark~3.4.8]{DHHR}.

\begin{proof}
This is an immediate consequence of Corollary~\ref{cor:Orlicz} and
Example~\ref{ex:Young_function_variable_Lebesgue;Delta2&complemented_function}.
\end{proof}

\begin{remark}
Let $Y:= L^{p(\cdot)}(T;X)$. Let us bound $\zeta(0,0)$ from below
using \eqref{eq:lowerboundforzeta00}. Note that by Example~\ref{ex:Young_function_variable_Lebesgue;Delta2&complemented_function}
one has that $K_{\Phi} = p_{+}$, $K_{\Psi} = p_{-}'$, and $h_{\Phi}=h_{\Psi} =0$,
so
$$
\zeta(0,0) \geq \frac{1}{6K_{\Psi}C_{K_{\Phi}, X}C_h} =
\frac{1}{3p_{-}'C_{p_{+}, X}},
$$
where an upper bound for $C_{p_{+},X}$ can be found using Remark
\ref{rem:oderofCKX}. From this one can obtain an upper bound for the UMD
constant using \eqref{eq:betavszeta}.
\end{remark}

In \cite{HouLiu} the analytic Radon--Nikodym (ARNP) and analytic
UMD (AUMD) properties are shown to hold for Musielak-Orlicz spaces
$L^{\Phi}(T)$ where $\Phi$ satisfies a condition which is slightly
more restrictive than $\Delta_2$. To end the paper we want to
state a related conjecture about spaces satisfying a randomized
version of UMD. In order to introduce it let $(\Omega', \A', \P')$
be a second probability space with a Rademacher sequence
$\varepsilon'=(\varepsilon_n')_{n\geq 1}$. A Banach space $X$ is said to be a
UMD$^{-}_{\rm PW}$ space if there is a $p\in [1, \infty)$ and a
constant $C\geq 0$ such that for all Paley--Walsh martingales $f$,
\[\|f\|_{L^p(\Omega;X)}\leq \|\varepsilon'*f\|_{L^p(\Omega\times \Omega';X)}.\]
This property turns out to be $p$-independent, and it gives a more general class of Banach spaces than the UMD spaces (see \cite{cox2018decoupling,CV1,CV2,Ga2}). For instance $L^1$ is a UMD$^{-}_{\rm PW}$ space.
\begin{conjecture}
Assume $\Phi:T\times [0,\infty)\to [0,\infty)$ is a Young function
such that $\Phi\in \Delta_2$. Then $L^{\Phi}(T)$ is UMD$^-_{\rm
PW}$.
\end{conjecture}
The conjecture is open also in the case $\Phi$ does not dependent on $T$. If $\Phi:[0,\infty)\to [0,\infty)$ is merely continuous, increasing to infinity and $\Phi(0)=0$ and satisfies $\Delta_2$, then the same question can be asked. However, in this case $L^{\Phi}(T)$ is not a Banach space, but only a quasi-Banach space. Some evidence for the conjecture can be found in \cite[Theorem 4.1]{CV2} and \cite[Theorem 1.1]{HM} where analogues of Lemma \ref{lem:UMDPhi} can be found (only $\Phi\in \Delta_2$ is needed in the proof). Doob's inequality plays a less prominent role for UMD$^-$ because of \cite[Lemma 2.2]{CV2}. Similar questions can be asked for the possibly more restrictive ``decoupling property'' of a quasi-Banach space $X$ introduced in \cite{CV1,CV2}.

\def\cprime{$'$}

\end{document}